\documentclass[reqno,11pt] {amsart}
\usepackage{amsmath, mathrsfs}
\usepackage[usenames, dvipsnames]{color}

\newtheorem{theorem}{Theorem}[section]
\newtheorem{lemma}[theorem]{Lemma}
\newtheorem{prop}[theorem]{Proposition}

\theoremstyle{definition}

\theoremstyle{remark}
\newtheorem{remark}[theorem]{Remark}

\numberwithin{equation}{section}

\newcommand{\re}{\text{Re}\,}

\newcommand{\bR}{\mathbb R}

\newcommand\md{\mathrm{d}}

\newcommand\cT{\mathcal{T}}

\newcommand{\Div}{\operatorname{div}}

\makeatletter
\def\dashint{\operatorname%
{\,\,\text{\bf--}\kern-.98em\DOTSI\intop\ilimits@\!\!}}
\makeatother

\renewcommand{\epsilon}{\varepsilon}

\begin{document}
\title [Multi-dimensional transport equation]{On a multi-dimensional transport equation with nonlocal velocity}

\author[H. Dong]{Hongjie Dong}
\address[H. Dong]{Division of Applied Mathematics, Brown University, 182 George Street, Box F, Providence, RI 02912, USA}
\email{Hongjie\_Dong@brown.edu}
\thanks{Email: Hongjie\_Dong@brown.edu, Tel: 1-4018637297, Fax: 1-4018631355. Hongjie Dong was partially supported by the National Science Foundation under agreement No. DMS-1056737.}




\begin{abstract}
We study a multi-dimensional nonlocal active scalar equation of the form $u_t+v\cdot \nabla u=0$ in $\bR^+\times \bR^d$, where $v=\Lambda^{-2+\alpha}\nabla u$ with $\Lambda=(-\Delta)^{1/2}$. We show that when $\alpha\in (0,2]$ certain radial solutions develop gradient blowup in finite time. In the case when $\alpha=0$, the equations are globally well-posed with arbitrary initial data in suitable Sobolev spaces.
\end{abstract}

\maketitle

\section{Introduction and main results}               \label{sec1}

The problem of finite-time singularity versus global regularity for active scalar equations with nonlocal velocities has attracted much attention in recent years. We refer the reader to a survey paper \cite{K10} and the references therein for some recent progress in this area.

In this paper, we study a multi-dimensional nonlocal active scalar equation of the form
\begin{equation}
                            \label{GQS}
u_t+v\cdot \nabla u=0
\end{equation}
for $x\in\bR^d$ and $t> 0$ with the initial data $u(0,x)=u_0(x)$ for $x\in\bR^d$.
Here $d\ge 1$ is the space dimension and $v$ is the velocity field determined by the constitutive relation
$v=\Lambda^{-2+\alpha}\nabla u$, where $\Lambda=(-\Delta)^{1/2}$ and $\alpha\in \bR$ is a constant. The equation becomes more singular as $\alpha$ increases.

In the 1D case when $\alpha=1$, Equation \eqref{GQS} is reduced to
\begin{equation}
                    \label{eq8.52}
u_t-(\mathcal H u)u_x=0,
\end{equation}
where $\mathcal H u$ is the Hilbert transform of $u$. It can be viewed as a one dimensional analogue of the 3D Euler equations in the vorticity formulation and the 2D surface quasi-geostrophic equation, as well as of the Birkhoff--Rott equations modeling the evolution of a vortex sheet with surface tension.
This equation was introduced by C\'{o}rdoba, C\'{o}rdoba, and Fontelos \cite{CCF05} (see also \cite{CCF06,K10}), where they showed that, for a generic family of initial data, local smooth solutions to \eqref{eq8.52} may blow up in finite time. It was also proved in  \cite{CCF05} that, by adding a dissipation term the equation
\begin{equation}
                    \label{eq8.52b}
u_t-(\mathcal Hu)u_x=-\Lambda^\beta u,
\end{equation}
is globally well-posed with positive $H^2$ initial data in the
case $\beta>1$, as well as in the case $\beta=1$ with small initial data. Later, in \cite{D08} the local well-posedness of \eqref{eq8.52b} in critical Sobolev spaces was studied in detail, and the global well-posedness was obtained for $\beta\ge 1$ with arbitrary initial data in suitable Sobolev spaces. In \cite{LR08} Li and Rodrigo proved formation of singularities in finite time of solutions when $\beta\in (0,1/2)$. In the case when $\beta\in [1/2,1)$, whether solutions of \eqref{eq8.52b} with smooth initial data may blow up in finite time remains to be an interesting open question.

A multi-dimensional generalization of \eqref{eq8.52} given by
\begin{equation}
                                \label{eq9.54}
u_t-\mathcal Ru\cdot \nabla u=0
\end{equation}
was considered in Balodis and C\'{o}rdoba \cite{BC07}, and finite-time blowup of solutions was proved. Here $\mathcal Ru=(\mathcal R_1 u,\ldots, \mathcal R_d u)$ is the Riesz transform of $u$.
When $d=2$, such result was also proved for a similar equation
in \cite{dongli2} independently.
Note that Equation \eqref{eq9.54} corresponds to \eqref{GQS} with $\alpha=1$, and can be viewed as a generalized model for the 2D surface quasi-geostrophic equation.

In the 2D case when $\alpha\in [1/2,1]$, Equation \eqref{GQS} (with a possible dissipation term) was first considered by Li and Rodrigo \cite{LR09}, and finite-time blowup of solutions was shown. In this case, the equation can be regarded as a generalization of the family of interpolating models between the surface quasi-geostrophic equation and the 2D Euler equation; See \cite{CFMR05, G08, CCW11,CIW} for some discussions.
In a recent paper \cite{Ch14}, Chae studied Equation \eqref{GQS} for general $\alpha$ in the multi-dimensional case. Among other results, he proved that the equation is locally well-posed in $H^s$ spaces for suitable $s$ when $\alpha\in (-d+1,2)$, and is globally well-posed for nonnegative initial data when $\alpha\in (-d+1,0)$.

The objective of this paper is twofold. First we show that when $\alpha=0$, the equation is still globally well-posed with arbitrary initial data in suitable Sobolev spaces. Second, $\alpha=0$ is the critical value for the global well-posedness of \eqref{GQS}, in the sense that if $\alpha\in (0,2]$ certain radial solutions develop gradient blowup in finite time.

Let us state our results more precisely. For the finite-time blowup of solutions, we consider initial profile $u_0$, which is smooth, radially symmetric, compactly supported, nonnegative, and attains its positive maximum at the origin.

\begin{theorem}[Finite-time blowup when $\alpha\in (0,2)$]
                            \label{thm1}
Let $\alpha\in (0,2)$. Then with any initial data $u_0$ satisfying the conditions above, the solution $u$ must develop gradient blowup in finite time.
\end{theorem}
In Theorem \ref{thm1}, the conditions that the data is compactly supported and attains its supremum at the origin are just for simplicity of the presentation. In general, the solution still blows up in finite time if the initial data is sufficiently concentrated near the origin. See \eqref{eq8.55} in Remark \ref{rem1.4}.

The proof of Theorem \ref{thm1} is in the spirit of the work \cite{CCF05}, where the authors established a weighted inequalities for the Hilbert transform in order to get a positive lower bound of the bilinear term. The following theorem is crucial for the proof of Theorem \ref{thm1} and might be of independent interest.
\begin{theorem}[A positive lower bound]
                        \label{thmbd}
Let $\alpha\in (0,2)$, $\delta\in (-\alpha,\alpha)$ be constants satisfying $\delta+\alpha<2$, and $f$ be a bounded even function on $\bR$ with bounded first derivatives.
Let
$$
(\cT f)(r)=\int_0^\infty f'(\rho)g(r/\rho)\rho^{1-\alpha}\,\md \rho,
$$
where $g=g_d$ is defined in \eqref{eq16.38} and \eqref{eq4.34} below.
Then there exists a constant $C_{d,\alpha,\delta}>0$ such that
\begin{equation*}
\int_0^\infty \frac {(\cT f)(r)f'(r) } {r^{1+\delta}} \,\md r \ge
C_{d,\alpha,\delta} \int_0^\infty \frac {\big(f(r)-f(0)\big)^2} {r^{1+\alpha+\delta}}  \,\md r.
\end{equation*}
\end{theorem}

Such inequality was known mostly in the case when $\alpha=1$.
It was first established in \cite{CCF05} when $d=1$ by a delicate argument involving contour integrals; see also \cite{CCF06} for some generalizations to non-even functions. In the multi-dimensional case, a similar inequality was systematically studied in \cite{BC07} for a class of possibly non-radial functions $f$ using spherical harmonic expansions. Independently, the case when $d=2$ and $\delta=0$ was treated in \cite{dongli2}.

The case when $\alpha\neq 1$ was first considered in \cite{LR09}, where
for $\alpha\in [1/2,1)$ and $d=2$, such inequality was obtained by using Polya's theorem on the cosine transform of convex functions. However, as pointed out in \cite{LR09} this argument no longer works when $d>2$ or $\alpha>0$ is small enough, for example $1/5$.

In terms of the ranges of $\alpha$ and $d$, Theorem \ref{thmbd} is more general than these in the literature. The proof of it is fairly transparent, and the key lemma is Lemma \ref{lem2.2}, which is a bit surprising. It shows that all the coefficients of Taylor's expansion of $g$ at $0$ are nonnegative whenever $\alpha\ge 0$.

Regarding the global well-posedness in the case $\alpha=0$, we assume that the initial data $u_0$ is sufficiently smooth and decays at the infinity.

\begin{theorem}[Global well-posedness when $\alpha=0$]
                            \label{thm2}
Let $\alpha=0$ and $d\ge 2$. Suppose that $u_0\in H^s$ for some $s>d/2+1$. When $d=2$, we also assume, for the local well posedness, that $u_0\in L_p$ for some $p\in (1,2)$.  Then there exists a unique smooth global solution to \eqref{GQS}. 
\end{theorem}

\begin{remark}
                            \label{rem1.4}
We note that our results can be extended to equations with fractional dissipations:
\begin{equation}
                            \label{eq3.56}
u_t+v\cdot \nabla u+\Lambda^\beta u=0.
\end{equation}
By using Theorem \ref{thmbd} and following the proof in \cite{LR09}, one can show the following finite-time blowup results for \eqref{eq3.56}. Under the assumptions of Theorem \ref{thm1}, let $\beta\in (0,\alpha/2)$ and $\delta\in (0,\alpha-2\beta)$. There exists a constant $N$ depending only on $d$, $\alpha$, $\beta$, and $\delta$, such that if
\begin{equation}
                        \label{eq8.55}
\int_0^\infty \frac {u_0(0)-u_0(r)} {r^{1+\delta}}\,\md r\ge N(1+\|u_0\|_{L_\infty}),
\end{equation}
then the solution to \eqref{eq3.56} with initial data $u_0=u_0(|x|)$ blows up in finite time. On the other hand, when $\alpha\in (0,2)$ and $\beta\in [\alpha,2]$, by modifying the argument of non-local moduli of continuity firstly appeared in \cite{KNV}, one can show that the solution is global in time. See, for instance, \cite{dongli14}.
We leave the details to the interested reader.
\end{remark}

The remaining part of the paper is organized as follows. In the next section, we rewrite the equation in polar coordinates. The proof of Theorem \ref{thmbd} is given in Section \ref{sec3} after we prove a lower bound of a certain function appeared in the Mellin transform of the nonlinear term. Finally, we complete the proofs of Theorems \ref{thm1} and \ref{thm2} in Section \ref{sec4}.

\section{Equation in polar coordinates}

It is clear that evolving from radially symmetric initial data, the solution of \eqref{GQS} is radial in its life span.
In this section, we derive the equation in polar coordinates. By the integral representation of the Riesz potential, the velocity is given by
$$
v(t,x)=\text{P.V.}\int_{\bR^d}\nabla u(t,y)G_{d,\alpha}(x-y)\,\md y
$$
for any $\alpha\in (2-d,2)$, where the kernel
$$
G_{d,\alpha}(x)=C_{d,\alpha}|x|^{-(d-2+\alpha)}
$$
and
$C_{d,\alpha}$
is a positive constant depending only on $d$ and $\alpha$.
When $d=1$,
$$
G_{d,\alpha}(x)=\left\{\begin{aligned}&-C_{d,\alpha}|x|^{1-\alpha}\quad
&\text{for}\,\,\alpha\in (0,1)\\
&-C_{d,\alpha}\log |x|\quad
&\text{for}\,\,\alpha= 1
\end{aligned}\right..
$$

Let $S_1$ be the unit sphere in $\bR^{d}$ and $\sigma_{d}$ be its surface area.
We recall that
$$
\sigma_d=\frac {2 \pi^{d/2}}{\Gamma(d/2)},
$$
where $\Gamma$ is the Gamma function.
For any radially symmetric solution $u(t,x)=u(t,|x|)$, by using polar coordinates the velocity $v$ can be expressed as
$$
v(t,x)=v(t,|x|)x/|x|
$$
and for $r\ge 0$,
\begin{equation}
                                            \label{eq16.39}
v(t,r)=\int_0^\infty \partial_\rho u(t,\rho)g(r/\rho)\rho^{1-\alpha}\,\md \rho=:\cT u(t,r),                             \end{equation}
where modulo a positive constant factor depending only on $d$ and $\alpha$, the function $g:[0,\infty)\to \bR$ is  defined by
\begin{align}
                            \label{eq16.38}
g(r)=g_d(r)&=\int_{S_1}\frac {y_1} {|r e_1-y|^{d-2+\alpha}}\,\md \sigma(y)\nonumber\\
&=\sigma_{d-1}\int_{0}^\pi \frac {\cos\theta\sin^{d-2}\theta} {(r^2+1-2r\cos\theta)^{(d-2+\alpha)/2}}\,\md \theta
\end{align}
when $d\ge 2$, and
\begin{equation}
                                \label{eq4.34}
g(r)=\left\{\begin{aligned}&|1+r|^{1-\alpha}-|1-r|^{1-\alpha}\quad
&\text{for}\,\,\alpha\in (0,1)\\
&\log |1+r|-\log |1-r|\quad
&\text{for}\,\,\alpha= 1\\
&|1-r|^{1-\alpha}-|1+r|^{1-\alpha}\quad
&\text{for}\,\,\alpha\in (1,2)
\end{aligned}\right.
\end{equation}
when $d=1$.
Under the radial symmetry assumption, Equation \eqref{GQS} is then reduced to
\begin{equation}
                                    \label{GQS2}
\partial_t u(t,r)+v(t,r)\partial_r u(t,r)=0
\end{equation}
in polar coordinates.

\section{Proof of Theorem \ref{thmbd}}
                        \label{sec3}

In this section, we give the proof of Theorem \ref{thmbd}. For $\lambda\in \bR$, define
\begin{equation}
                            \label{eq12.38}
H_1(\lambda):=\int_0^\infty r^{i\lambda-2+\frac \alpha 2-\frac \delta 2}g(r)\,\md r
\end{equation}
and
\begin{equation}
                        \label{eq7.53}
H(\lambda)=
\left(\big(\alpha+\delta\big)^2/4+\lambda^2\right) H_1(\lambda).
\end{equation}
The motivation for introducing these two functions will be clear from the proof of Theorem \ref{thm1} below.
\begin{prop}
                \label{prop2.1}
Let $\alpha\in (0,2)$, $\delta\in (-\alpha,\alpha)$, and $g$ and $H$ be the functions defined in \eqref{eq16.38} (or \eqref{eq4.34} when $d=1$) and \eqref{eq7.53}, respectively. Then for any $\lambda\in \bR$, we have
\begin{equation*}
\re H(\lambda)\ge C_{d,\alpha,\delta}>0.
\end{equation*}
\end{prop}

In the proposition above, we require $\delta<\alpha$ so that the integral in \eqref{eq12.38} is convergent. The condition that $\delta>-\alpha$ is only to guarantee that the factor
$$
\big(\alpha+\delta\big)^2/4+\lambda^2
$$
is strictly positive for any $\lambda\in \bR$, and can certainly be relaxed.

To prove the proposition, we begin with analyzing some properties of $g$. 
Recall that when $d\ge 2$,
\begin{align}
                            \label{eq16.38b}
g(r)=\sigma_{d-1}\int_{0}^\pi \frac {\cos\theta\sin^{d-2}\theta} {(r^2+1-2r\cos\theta)^{(d-2+\alpha)/2}}\,\md \theta
\end{align}
for $r\in [0,\infty)$, and when $d=1$, $g$ is defined in \eqref{eq4.34}. It is smooth in $[0,1)\cup (1,\infty)$ with a possible singularity at $1$. The singularity is of order $|r-1|^{1-\alpha}$ when $\alpha>1$ and of order $\log |r-1|$ when $\alpha=1$. We can extend $g$ to $(-\infty,\infty)$ by using \eqref{eq16.38b} (or \eqref{eq4.34} when $d=1$). It is easily seen that, as a function in $(-1,1)$,  $g$ is odd, real analytic, and its Taylor's expansion at $0$ has the form:
\begin{equation}
                    \label{eq9.52}
g(r)=\sum_{n=0}^\infty a_{2n+1}r^{2n+1}.
\end{equation}
The following lemma is a bit surprising, which reads that all the coefficients $a_{2n+1}$ above are nonnegative whenever $\alpha\ge 0$.
\begin{lemma}
                            \label{lem2.2}
Let $g:(-\infty,\infty)\to \bR$ be the function defined in \eqref{eq16.38b}  (or \eqref{eq4.34} when $d=1$). Then as a function in $(-1,1)$, the following hold. When $\alpha=0$, we have $a_1>0$ and $a_{2n+1}=0$ for $n\ge 1$. When $\alpha>0$, we have $a_{2n+1}>0$ for any $n\ge 0$.
\end{lemma}
\begin{proof}
When $d=1$, it suffices for us to apply the Taylor expansion to \eqref{eq4.34}. In the sequel, we assume $d\ge 2$. We shall calculate the coefficients $a_{2n+1}$ explicitly. Denote
$$
A(r,\theta)=(r^2+1-2r\cos\theta)^{1/2}.
$$
Integrating by parts, we can rewrite $g$ into a more convenient form:
\begin{equation}
                        \label{eq11.07}
g(r)=\sigma_{d-1}\int_{0}^\pi \frac {(\sin^{d-1}\theta)'} {(d-1)A^{d-2+\alpha}(r,\theta)}\,\md \theta
:=\sigma_{d-1}r\frac {d-2+\alpha} {d-1}f(r,d,\alpha),
\end{equation}
where
$$
f(r,d,\alpha)=\int_{0}^\pi \frac {\sin^{d}\theta} {A^{d+\alpha}(r,\theta)}\,\md \theta.
$$

{\em Case 1: $\alpha=0$.} For fixed $r>1$, $|x-re_1|^{-d}$ is a harmonic function on the unit ball in $\bR^{d+2}$. Consequently, by the mean value theorem, we get for $r>1$
$$
f(r,d,0)=\int_{0}^\pi \frac {\sin^{d}\theta} {A^{d}(r,\theta)}\,\md \theta=\frac {\sigma_{d+2}} {\sigma_{d+1}}r^{-d}.
$$
For any $r\in (0,1)$, since $A(1/r,\theta)=A(r,\theta)/r$, it follows that
$$
f(r,d,0)=r^df(1/r,d,0)=\frac {\sigma_{d+2}} {\sigma_{d+1}}.
$$
When $r=0$, we also have
$$
f(0,d,\alpha)=\int_{0}^\pi {\sin^{d}\theta} \,\md \theta=\frac {\sigma_{d+2}} {\sigma_{d+1}}.
$$
Thus in this case, we obtain
$$
g(r)=\sigma_{d-1}r\frac {d-2}{d-1} \frac {\sigma_{d+2}} {\sigma_{d+1}}
=r(d-2)\frac {\pi^{d/2}}{\Gamma\big((d+2)/2\big)}=r\frac {d-2} d\sigma_d.
$$
for any $r\in (-1,1)$.

{\em Case 2: $\alpha>0$.} It is obvious that $f(r,d,\alpha)$ is infinitely differentiable (and analytic) in $r\in (-1,1)$. We compute
$$
\partial_r f(r,d,\alpha)=-\int_{0}^\pi (d+\alpha)\frac {\sin^{d}\theta(r-\cos\theta)} {A^{d+\alpha+2}(r,\theta)}\,\md \theta,
$$
and
\begin{align*}
&\partial_r^2 f(r,d,\alpha)=-\int_{0}^\pi (d+\alpha)\frac {\sin^{d}\theta} {A^{d+\alpha+2}(r,\theta)}\,\md \theta\\
&\quad+\int_{0}^\pi (d+\alpha)(d+\alpha+2)\frac {\sin^{d}\theta(r-\cos\theta)^2} {A^{d+\alpha+4}(r,\theta)}\,\md \theta.
\end{align*}
Since
$$
(r-\cos\theta)^2=A^2(r,\theta)-\sin^2\theta,
$$
we get
\begin{align}
&\partial_r^2 f(r,d,\alpha)\nonumber\\
&=\int_{0}^\pi (d+\alpha)(d+\alpha+1)\frac {\sin^{d}\theta} {A^{d+\alpha+2}(r,\theta)}\,\md \theta\nonumber\\
&\quad-\int_{0}^\pi (d+\alpha)(d+\alpha+2)\frac {\sin^{d+2}\theta} {A^{d+\alpha+4}(r,\theta)}\,\md \theta\nonumber\\
                            \label{eq8.57}
&=(d+\alpha)\big[(d+\alpha+1) f(r,d,\alpha+2)-(d+\alpha+2)f(r,d+2,\alpha+2)\big].
\end{align}
Note that
$$
f(0,d,\alpha)=\frac {\sigma_{d+2}} {\sigma_{d+1}}=\frac{\pi^{1/2}\Gamma\big(\frac{d+1} 2\big)}{\Gamma\big(\frac{d+2} 2\big)}=B\left(\frac 1 2,\frac{d+1} 2\right),
$$
where $B(\cdot,\cdot)$ is the Beta function.
By induction, from the recurrence relation \eqref{eq8.57} it is easy to check that
\begin{align}
                                    \label{eq9.04}
\partial^{2n}_r f(0,d,\alpha)=&B\left(\frac 1 2,\frac{d+1} 2\right)(2n-1)!!\,\alpha(\alpha+2)\cdots(\alpha+2n-2)\nonumber\\
&\quad \cdot \frac {(d+\alpha)(d+\alpha+2)\cdots(d+\alpha+2n-2)}
{(d+2)(d+4)\cdots(d+2n)}
\end{align}
for any integer $n\ge 0$.
Therefore, using \eqref{eq11.07} and \eqref{eq9.04},
\begin{align}
                                    \label{eq9.30}
a_{2n+1}&=\frac{\sigma_{d-1}}{(2n)!}\, \frac{d-2+\alpha} {d-1}\,\partial^{2n}_r f(0,d,\alpha)\nonumber\\
&=\sigma_d \frac{\alpha(\alpha+2)\cdots(\alpha+2n-2)}{(2n)!!}\nonumber\\
&\quad \cdot \frac {(d+\alpha-2)(d+\alpha)(d+\alpha+2)\cdots(d+\alpha+2n-2)}
{d(d+2)(d+4)\cdots(d+2n)}
\end{align}
for any integer $n\ge 0$. It is worth noting that from the expression \eqref{eq9.30} (and \eqref{eq4.34} in 1D case), one can easily see that asymptotically
\begin{equation}
                                    \label{eq10.23}
a_{2n+1}\sim n^{\alpha-2}\quad \text{as}\quad n\to \infty.
\end{equation}
Thus \eqref{eq9.52} is convergent for any $r\in (-1,1)$. The lemma is proved.
\end{proof}

We are ready to give the proof of Proposition \ref{prop2.1}.

\begin{proof}[Proof of Proposition \ref{prop2.1}]
We first note that
$$
g(1/r)=r^{d-2+\alpha}g(r).
$$
By a change of variables, we get
\begin{align*}
H_1(\lambda)&=\left(\int_0^1+\int_1^\infty\right) r^{i\lambda-2+\frac \alpha 2-\frac \delta 2}g(r)\,\md r\\
&=\int_0^1r^{i\lambda-2+\frac \alpha 2-\frac \delta 2}g(r)\,\md r+\int_0^1
r^{-i\lambda-\frac \alpha 2+\frac \delta 2}g(1/r)\,\md r\\
&=\int_0^1\big(r^{i\lambda-2+\frac \alpha 2-\frac \delta 2}+
r^{-i\lambda+d-2+\frac \alpha 2+\frac \delta 2}\big)g(r)\,\md r.
\end{align*}
It then follows from Lemma \ref{lem2.2} that
\begin{align*}
H_1(\lambda)
&=\int_0^1\big(r^{i\lambda-2+\frac \alpha 2-\frac \delta 2}+
r^{-i\lambda+d-2+\frac \alpha 2+\frac \delta 2}\big)\sum_{n=0}^\infty a_{2n+1}r^{2n+1}\,\md r\\
&=\sum_{n=0}^\infty a_{2n+1} \int_0^1\big(r^{i\lambda-2+\frac \alpha 2-\frac \delta 2}+
r^{-i\lambda+d-2+\frac \alpha 2+\frac \delta 2}\big) r^{2n+1}\,\md r\\
&=\sum_{n=0}^\infty a_{2n+1} \left(\frac 1 {i\lambda+2n+\frac \alpha 2-\frac \delta 2}+\frac 1 {-i\lambda+d+2n+\frac \alpha 2+\frac \delta 2}\right),
\end{align*}
where $a_{2n+1}>0$ for any $n\ge 0$.
Here we used the condition $\delta<\alpha$ so that the integrals are convergent and the condition $\alpha<2$ so that by \eqref{eq10.23} the summation is absolutely convergent, which justifies the second equality by using the Fubini theorem. Therefore, by \eqref{eq7.53},
\begin{align}
&\re H(\lambda)\nonumber\\
&=\left(\big(\alpha+\delta\big)^2/4+\lambda^2\right)\re H_1(\lambda)\nonumber\\
&=\left(\big(\alpha+\delta\big)^2/4+\lambda^2\right)\nonumber\\
&\quad\cdot\sum_{n=0}^\infty a_{2n+1} \left(\frac {2n+\frac \alpha 2-\frac \delta 2} {\lambda^2+(2n+\frac \alpha 2-\frac \delta 2)^2}+\frac {d+2n+\frac \alpha 2+\frac \delta 2} {\lambda^2+(d+2n+\frac \alpha 2+\frac \delta 2)^2}\right).       \label{eq10.43}
\end{align}
Since for each $n$,
$$
\frac {\big(\alpha+\delta\big)^2/4+\lambda^2} {\lambda^2+(d+2n+\frac \alpha 2+\frac \delta 2)^2}
$$
is increasing with respect to $\lambda\in [0,\infty)$ and $\delta>-\alpha$, we see that
$$
\re H(\lambda)\ge \frac {\big(\alpha+\delta\big)^2} 4 \sum_{n=0}^\infty \frac {a_{2n+1}} {d+2n+\frac \alpha 2+\frac \delta 2}>0.
$$
The proposition is proved.
\end{proof}
\begin{remark}
By using \eqref{eq10.43} and \eqref{eq10.23}, it is easily seen that asymptotically
$$
\re H(\lambda)\sim |\lambda|^\alpha\quad \text{as}\quad |\lambda|\to \infty.
$$
However, we will not use this fact in the proof below.
\end{remark}

Now we are in the position to prove Theorem \ref{thmbd}.

\begin{proof}[Proof of Theorem \ref{thmbd}]
Denote
\begin{align*}
I:&=\int_0^\infty \frac {(\cT f)(r)f'(r) } {r^{1+\delta}} \,\md r\\
&=\int_0^\infty \frac {(\cT f)(r)} {r^{1-\frac \alpha 2+\frac \delta 2}}\cdot
\frac{f'(r) } {r^{\frac \alpha 2 +\frac \delta 2-1}} \,\frac{\md r} r.
\end{align*}
By the Parseval identity for the Mellin transform, we have
\begin{equation}
                                \label{eq1.54}
I=\frac 1 {2\pi}\int_{-\infty}^\infty A(\lambda)\overline{B(\lambda)}\,\md \lambda,
\end{equation}
where
\begin{align*}
A(\lambda)=\int_0^\infty r^{i\lambda-\frac \alpha 2-\frac \delta 2}f'(r)\,\md r,\quad
B(\lambda)=\int_0^\infty r^{i\lambda-2+\frac \alpha 2-\frac \delta 2}\cT f(r)\,\md r.
\end{align*}
Using \eqref{eq16.39}, the Fubini theorem, and a change of variables, we have
\begin{align*}
B(\lambda)&=\int_0^\infty \int_0^\infty r^{i\lambda-2+\frac \alpha 2-\frac \delta 2}f'(\rho)g(r/\rho)\rho^{1-\alpha}\,\md \rho\,\md r\\
&=\int_0^\infty \rho^{1-\alpha}f'(\rho)\left( \int_0^\infty r^{i\lambda-2+\frac \alpha 2-\frac \delta 2}g(r/\rho)\,\md r\right)\,\md \rho\\
&=H_1(\lambda)\int_0^\infty \rho^{i\lambda-\frac \alpha 2-\frac \delta 2} f'(\rho)\,\md \rho,
\end{align*}
where $H_1(\lambda)$ is defined in \eqref{eq12.38}.
We then integrate by parts to get
\begin{align*}
B(\lambda)&=(\frac \alpha 2+\frac \delta 2-i\lambda)H_1(\lambda)\int_0^\infty r^{i\lambda-\frac \alpha 2-\frac \delta 2-1}\big(f(r)-f(0)\big)\,\md r,\\
A(\lambda)&=(\frac \alpha 2+\frac \delta 2-i\lambda)\int_0^\infty r^{i\lambda-\frac \alpha 2-\frac \delta 2-1}\big(f(r)-f(0)\big)\,\md r.
\end{align*}
Here we used the fact that $\delta+\alpha<2$ so that the boundary terms at $0$ vanish.
These together with \eqref{eq1.54} yield
\begin{align*}
I&=\frac 1 {2\pi} \int_{-\infty}^\infty \left(\frac {(\alpha+\delta)^2} 4+\lambda^2\right)\overline{H_1(\lambda)}
\left|\int_0^\infty r^{i\lambda-\frac \alpha 2-\frac \delta 2-1}\big(f(r)-f(0)\big)\,\md r\right|^2\,\md \lambda\\
&=\frac 1 {2\pi} \int_{-\infty}^\infty \overline{H(\lambda)}
\left|\int_0^\infty r^{i\lambda-\frac \alpha 2-\frac \delta 2-1}\big(f(r)-f(0)\big)\,\md r\right|^2\,\md \lambda.
\end{align*}
Since $I$ is a real valued function, we have
$$
I=\frac 1 {2\pi} \int_{-\infty}^\infty \re \big(H(\lambda)\big)
\left|\int_0^\infty r^{i\lambda-\frac \alpha 2-\frac \delta 2-1}\big(f(r)-f(0)\big)\,\md r\right|^2\,\md \lambda.
$$
By Proposition \ref{prop2.1},
$$
\re \big(H(\lambda)\big)\ge C_{d,\alpha,\delta}>0
$$
for any $\lambda\in \bR$.
We then obtain by applying the Parseval identity again that
\begin{align*}
I&\ge \frac {C_{d,\alpha,\delta}} {2\pi} \int_{-\infty}^\infty
\left|\int_0^\infty r^{i\lambda-\frac \alpha 2-\frac \delta 2-1}\big(f(r)-f(0)\big)\,\md r\right|^2\,\md \lambda\\
&=C_{d,\alpha,\delta}\int_0^\infty \frac {\big(f(r)-f(0)\big)^2}{r^{1+\alpha+\delta}}\,\md r.
\end{align*}
The theorem is proved.
\end{proof}

\section{Proofs of Theorems \ref{thm1} and \ref{thm2}}
                                    \label{sec4}
In the last section, we complete the proofs of Theorems \ref{thm1} and \ref{thm2}.

\begin{proof}[Proof of Theorem \ref{thm1}]
For the case when $\alpha\in (0,1]$ (regular velocity), the local well-posedness of Equation \eqref{GQS} is quite standard. See, for instance, \cite{dongli14} and \cite{Ch14}. For $\alpha\in (1,2]$ (singular velocity), we refer the reader  to \cite{CCCFW} and \cite{Ch14}, where a novel commutator estimate was applied to non-local transport equations with the velocity given by convolution with odd kernels.

Next we show the finite-time blowup of the solution. We suppose that $u_0$ is supported in a ball $B_R(0)$ for some $R>0$. Since $u\ge 0$ and radial, it is not difficult to see from the representation formula of $v$ that the support of $u(t,\cdot)$ shrinks as $t$ increases\footnote{We note that this is true even in the non-radial case if the support is replaced by the convex hull of the support.}. Therefore, for any $t>0$ the support of $u(t,\cdot)$ lies inside $B_R(0)$. Moreover, since $v(t,0)=0$ for any $t\ge 0$ by the radial symmetry, $u(t,0)=u_0(0)$ for any $t\ge 0$. In the case when $\alpha=2$, the equation is reduced to the Burgers' equation
$$
u_t(t,r)+(\partial_r u(t,x))^2=0,
$$
for which the finite-time blowup is well known. For $\alpha\in (0,2)$, let $\delta\in (0,\min\{\alpha,2-\alpha\})$, which certainly implies that $\delta<1$.
We denote
$$
I(t):=\int_0^L \frac {u(t,0)-u(t,r)} {r^{1+\delta}} \,\md r.
$$
We multiply Equation \eqref{GQS2} by the weight $r^{-1-\delta}$ and integrate in $r\in (0,L)$.
By Theorem \ref{thmbd}, we have
\begin{align*}
\frac d {dt} I(t)&=\int_0^L \frac {-\partial_t u(t,r)} {r^{1+\delta}} \,\md r= \int_0^L \frac {\cT u(t,r) \partial_r u(t,r)} {r^{1+\delta}} \,\md r \\
& =\int_0^\infty \frac {\cT u(t,r) \partial_r u(t,r)} {r^{1+\delta}} \,\md r  \\
& \ge C_{d,\alpha,\delta} \int_0^\infty \frac {(u(t,0)-u(t,r))^2} {r^{1+\alpha+\delta}} \,\md r \\
& \ge C_{d,\alpha,\delta} \int_0^L \frac {(u(t,0)-u(t,r))^2} {r^{1+\alpha+\delta}} \,\md r \\
& \ge \tilde C_{d,\alpha,\delta} \left(\int_0^L \frac {u(t,0)-u(t,r)} {r^{1+\delta}}
\,\md r\right)^2\\
&=\tilde C_{d,\alpha,\delta} (I(t))^2,
\end{align*}
where in the last inequality we used the fact that $\delta<\alpha$ and H\"older's inequality. This implies that
$I(t)$ blows up in finite time. Now since
$$
I(t) \le \frac {L^{1-\delta}} {1-\delta} \|\partial_r u(t,\cdot) \|_{L^\infty},
$$
we conclude that $\| \partial_r u(t,\cdot) \|_{L^\infty}$ must blow up in finite time. The theorem is proved.
\end{proof}

\begin{proof}[Proof of Theorem \ref{thm2}]
Again for the proof of the local well-posedness, we refer the reader to, for instance, \cite{dongli14} and \cite{Ch14}. Note that the condition $u_0\in L_p$ for some $p\in (1,2)$ when $d=2$ is used to
 control the low frequency part in the contraction argument.

We now prove the global regularity. Multiplying both sides of \eqref{GQS} by $u|u|^{p-2}$ and integrating in $x$, it is easy to see that
\begin{align}
\frac {d} {dt} \| u(t,\cdot)\|_{L_p} \le \| \Div v(t,\cdot) \|_{L_\infty} \| u(t,\cdot) \|_{L_p}.
\label{Mar6_e13}
\end{align}
In the sequel, $N$ is a constant depending only on $d$ and $s$, which may vary from line to line. For $k=0,\ldots, s$, we differentiate \eqref{GQS} with respect to the $x$ variable $k$ times, multiply it by $\partial_x^k u$, and use integration by parts, Leibniz's rule, and H\"older's inequality to get
\begin{align}
&\frac 1 2\frac d {dt} \Bigl( \| \partial_x^k u(t,\cdot) \|_{L_2}^2 \Bigr)\notag\\
& = \int_{\mathbb R} \partial_x^k( v \cdot \nabla u ) \partial_x^k u \,\md x \notag \\
& \le N\| \Div v \|_{L_\infty} \| \partial_x^k u(t,\cdot) \|_{L_2}^2 \notag \\
& \quad +N \| \partial_x^k u(t,\cdot) \|_{L_2}\sum_{l=1}^k \|
\partial_x^{l}v(t,\cdot)
\|_{L_{\frac{2k}{l-1}}} \cdot \|
\partial_x^{k-l}\nabla u(t,\cdot)
\|_{L_{\frac{2k}{k-l+1}}}. \label{Ma6_e11}
\end{align}
By using the Gagliardo--Nirenberg inequality and the boundedness of Riesz operators in $L_p,p\in (1,\infty)$, we have for any $2\le l \le k$,
\begin{align*}
\| \partial_x^{l}v(t,\cdot) \|_{L_{\frac{2k}{l-1}}}
&\le N\| \partial_x v(t,\cdot) \|_{L_\infty}^{1-\frac{l-1}{k}}
\cdot \| \partial_x^{k+1} v(t,\cdot)  \|_{L_2}^{\frac{l-1}{k}} \\
&\le N\| \partial_x v(t,\cdot) \|_{L_\infty}^{1-\frac{l-1}{k}}
\cdot \| \partial_x^k u(t,\cdot)  \|_{L_2}^{\frac{l-1}{k}}
\end{align*}
and
\begin{align*}
\| \partial_x^{k-l} \nabla u(t,\cdot) \|_{L_{\frac{2k}{k-l+1}}}
&\le N \| \partial_x^{k-l+2}v(t,\cdot) \|_{L_{\frac{2k}{k-l+1}}}\\
&\le N\| \partial_x v(t,\cdot) \|_{L_\infty}^{1-\frac{k-l+1}{k}}
\cdot \| \partial_x^{k+1}  v (t,\cdot)\|_{L_2}^{\frac{k-l+1}{k}}\\
&\le N\| \partial_x v(t,\cdot) \|_{L_\infty}^{1-\frac{k-l+1}{k}}
\cdot \| \partial_x^{k}   u (t,\cdot)\|_{L_2}^{\frac{k-l+1}{k}}.
\end{align*}
Inserting the above estimates into \eqref{Ma6_e11},  we obtain
\begin{align*}
\frac{d}{dt} \Bigl( \| \partial_x^k u(t,\cdot) \|_{L_2}^2 \Bigr) & \le
N\| \partial_x v(t,\cdot) \|_{L_\infty} \cdot
\|\partial_x^k u(t,\cdot)\|_{L_2}^2.
\end{align*}
This together with \eqref{Mar6_e13} yields
\begin{align}
&\frac{d}{dt} \Bigl( \| u(t,\cdot) \|_{H^s}+\| u(t,\cdot) \|_{L_p} \Bigr)\nonumber\\
                            \label{eq11.02}
& \le
N\| R\otimes R u(t,\cdot) \|_{L_\infty} \cdot \Bigl(
 \| u(t,\cdot) \|_{H^s}+\| u(t,\cdot) \|_{L_p}\Bigr).
\end{align}
From the logarithmic type inequality
$$
\|R\otimes R u\|_{L_\infty}\leq N \|u\|_{L_\infty}\log(e+\|u\|_{H^s})=N \|u_0\|_{L_\infty}\log(e+\|u\|_{H^s})
$$
and \eqref{eq11.02}, we get the double exponential bound of the norms
$$
\| u(t,\cdot) \|_{H^s}+\| u(t,\cdot) \|_{L_p}\le \text{exp}(e^{Nt\|u_0\|_{L_\infty}})\big(\| u_0\|_{H^s}+\| u_0 \|_{L_p}\big)
$$
by using Gronwall's inequality twice.
\end{proof}

\begin{remark}
By a straightforward modification of the proof, we can extend Inequality \eqref{eq11.02} to the case when $\alpha\in [0,1]$:
\begin{align*}
&\frac{d}{dt} \Bigl( \| u(t,\cdot) \|_{H^s}+\| u(t,\cdot) \|_{L_p} \Bigr)\nonumber\\
& \le
N\| R\otimes R \Lambda^{\alpha} u(t,\cdot) \|_{L_\infty} \cdot \Bigl(
 \| u(t,\cdot) \|_{H^s}+\| u(t,\cdot) \|_{L_p}\Bigr).
\end{align*}
This together with the Sobolev imbedding theorem actually implies a regularity criterion which reads that $u$ is regular up to time $T$ if and only if
$$
\int_0^T \| R\otimes R \Lambda^{\alpha} u(t,\cdot) \|_{L_\infty}\,\md t <\infty,
$$
and in that case we have
\begin{equation}
                                    \label{eq9.30bb}
\| u(t,\cdot) \|_{H^s}+\| u(t,\cdot) \|_{L_p}\le e^{N\int_0^T \| R\otimes R \Lambda^{\alpha}u(t,\cdot) \|_{L_\infty}\,\md t }\big(\| u_0\|_{H^s}+\| u_0 \|_{L_p}\big)
\end{equation}
for any $t\in [0,T]$. This regularity criterion result was contained in Theorem 1.1 (ii) of \cite{Ch14} with a different proof. However, instead of the single exponential bound \eqref{eq9.30bb}, a double exponential bound was proved in \cite{Ch14}. The single exponential bound here enable us to obtain the global regularity in the case $\alpha=0$ as shown above.

It is also worth noting that in view of \eqref{eq9.30bb}, for $\alpha\in (0,1)$, the $C^\alpha$ norm of the solution in Theorem \ref{thm1} must blow up in finite time.
\end{remark}



\end{document}